\documentclass{amsart}
\newtheorem{theorem}{Theorem}[section]

\usepackage{latexsym}
\usepackage{amssymb,amsmath,amsfonts}
\usepackage[colorlinks=true]{hyperref}
\usepackage{graphicx}
\usepackage{array}
\usepackage{booktabs}

\usepackage{float}
\usepackage{subfig}
\usepackage{adjustbox}
\usepackage{longtable} 
\usepackage {setspace}
\usepackage{multicol, multirow}
\usepackage{bigstrut}
\usepackage{array}
\usepackage{tabularx}
\usepackage{amsmath}
\usepackage{amssymb}
\usepackage{amsthm}
\usepackage{amsfonts}
\usepackage{booktabs}

\usepackage{makecell}
\usepackage{cite}
\usepackage{amssymb,amsmath,amsfonts, amsthm, mathrsfs}
\usepackage{graphicx,fancyhdr}
\newtheorem{ass}[theorem]{Assumption}
\newtheorem{alg}[theorem]{Algorithm}

\newcommand{\quotes}[1]{``#1''}

\newtheorem{lemma}[theorem]{Lemma}

\theoremstyle{definition}
\newtheorem{definition}[theorem]{Definition}

\newtheorem{example}[theorem]{Example}

\theoremstyle{remark}
\newtheorem{remark}[theorem]{Remark}

\numberwithin{equation}{section}
\global\parskip 6pt \oddsidemargin
0.1cm \evensidemargin 0.4cm \headheight 0.1cm \topmargin -2.0cm
\begin{document}
\title
[Pesudomonotone equilibirum problem ]
{A new inertial condition on the subgradient extragradient method for solving pseudomonotone equilibrium problem}
\author{$^{1}$Chinedu  Izuchukwu,  $^{2}$Grace  Nnennaya Ogwo,$^{3,4}$Bertin Zinsou}

\keywords{Equilibrium problems; pseudomonotone operator;  inertial technique; subgradient extragradient method; inertial condition.\\
{\rm 2010} {\it Mathematics Subject Classification}: 47H09; 47H10; 49J20; 49J40\\\\
{$^{1,2,3}$School of Mathematics, University of the Witwatersrand, Private Bag 3, Johannesburg, 2050, South Africa.}\\
{$^4$National Institute for Theoretical and Computational Sciences (NITheCS), South Africa}\\
{$^{1}$chinedu.izuchukwu@wits.ac.za,}\\
{$^{2}$grace.ogwo@wits.ac.za}\\
{$^{3,4}$bertin.zinsou@wits.ac.za}}

\begin{abstract}
\noindent In this paper we study the pseudomonotone equilibrium problem. We consider a new  inertial condition for the  subgradient extragradient method with self-adaptive step size for approximating a solution of the equilibrium problem  in a real Hilbert space. Our proposed method contains inertial factor with new conditions that only depend on the iteration coefficient. We obtain a weak convergence result of the proposed method under  weaker conditions on the inertial factor than many existing conditions in the literature.  Finally,  we present some numerical experiments for our proposed method in comparison with  existing methods in the literature. Our result improves, extends and generalizes several existing results in the literature. 
\end{abstract}
	
\maketitle\section{Introduction}
\noindent Let $\mathcal{C}$ be a nonempty closed and convex subset of a real Hilbert space $\mathcal{H}.$ The equilibrium problem (EP) introduced by Blum and Oettli \cite{blum} is the problem of finding a point $x^*\in\mathcal{C}$ such that 
\begin{align}\label{ep}
F(x^*,y)\geq 0, ~\forall~ y\in \mathcal{C},
\end{align} where $F:\mathcal{C}\times\mathcal{C}\to\mathbb{R}$ is a bifunction. Any point $x^*\in \mathcal{C}$ that solves this problem is called an equilibrium point of $F.$ We denote by $EP(F,\mathcal{C})$ the solution set of  Problem \eqref{ep}.
\begin{definition}
		
\noindent A bifunction $F:\mathcal{C}\times \mathcal{C}\to \mathbb{R}$ is said to be 
\begin{enumerate}
\item [(i)]  {\it strongly monotone on $\mathcal{C}$}, if there exists  a constant $c>0$ such that
\begin{align*}
F(x,y)+F(y,x)\leq -c\|x-y\|^2,~~~ \forall~x,y\in \mathcal{C},
\end{align*}
\item[(ii)] {\it monotone on $\mathcal{C}$}, if
\begin{align*}F(x,y)+F(y,x)\leq 0,~\forall~ x,y \in \mathcal{C},
\end{align*}
\item[(iii)]{\it pseudomonotone on $\mathcal{C}$}, if \begin{align*}
F(x,y)\geq0\implies F(y,x)\leq 0,~\forall~x,y\in \mathcal{C},
			\end{align*}
\item[(iv)] satisfying a Lipschitz-like condition on $\mathcal{C}$ if there exist constants $a_1>0$ and $a_2>0$ such that
\begin{align*}F(x,y)+F(y,w)\geq F(x,w)-a_1\|x-y\|^2-a_2\|y-w\|^2,~~~\forall x,y,w\in\mathcal{C}.\end{align*}
\end{enumerate}
\end{definition}
\noindent   We observe that $(i)\implies (ii)\implies(iii)$ but the converses are not always true (see \cite{rp} and other references therein).
	
\noindent The EP \eqref{ep} has received a lot of attention from several  researchers due to the fact that it unifies in a simple form several mathematical models such as optimization problem, fixed point problem, convex minimization problem, Nash equilibrium, variational inequality problem, saddle point problem, among others (see \cite{ii1,ii2,manb,mana} and other references therein). Many authors have proposed and studied several iterative methods for approximating  solutions of  EP \eqref{ep}  and other related optimization problems (see \cite{YCQY,he1,he2,jwp,xq, mou} and other references therein). In  1976, Kopelevich \cite{Kor} introduced the extragradient method for solving saddle point problem.  Quoc et al. \cite{qu} extended the extragradient method to solve the EP \eqref{ep} in a finite dimensional space. This result was later extended to an infinite dimensional Hilbert space by Vinh and Muu \cite{vm}.  They obtained a weak convergence result  under the assumptions that the equilibrium bifunction is pseudomonotone and satisfies the Lipschitz-like condition.  When using this method, one needs to solve two strongly convex optimization problems in the feasible set $\mathcal{C}$ per iteration. This is a major drawback on the extragradient method and could cause the method to be computationally expensive if the set $\mathcal{C}$ is not  simple. To circumvent this limitation, Rehman et al. \cite{hur}  extended the subgradient extragradient method in \cite{Reich} from solving variational inequalities to solving EP \eqref{ep}. The major advantage of the subgradient extragradient method over the extragradient method is that  the second convex optimization problem is onto a half-space which has a closed form solution. Thus, its computational complexity is less expensive than the extragradient method.

\hfill
\noindent  The inertial technique which originated from the heavy ball method of a second order dissipative dynamical system in time was derived by Polyak \cite{pol}. It is one of the  techniques often employed by authors to improve the convergence speed  of iterative methods when solving optimization problems.  This is due to the  fact that it increases the rate of convergence of iterative schemes.  There has been an increase interest in studying inertial type algorithms  for solving optimization problems (see \cite{Dong, att, bot2, PA2, PA4, mouda}), and one key interest in these studies is how to improve the conditions on the inertial factor \cite{Dong}.  In 2003, Moudafi \cite{mouda}  proposed an inertial algorithm for solving the EP \eqref{ep}: Find $x_{n+1}\in\mathcal{C}$ such that 
\begin{eqnarray*}
	F(x_{n+1},x)+\lambda^{-1}_n\langle x_{n+1}-y_n,x-x_{n+1}\rangle\ge -\epsilon_n,~~\forall~~x\in\mathcal{C},
\end{eqnarray*} where $y_n:=x_n+\theta_n(x_n-x_{n-1})$, $~\{\lambda_n\},\{\epsilon_n\}$  are sequences of nonnegative real numbers and the  inertial factor $\theta_n$ satisfies
\begin{eqnarray}\label{con}
0\le \theta_n\le \theta<1~~ \forall n\geq 1, ~~~~\sum\limits_{n=1}^{\infty}\theta_n\|x_n-x_{n-1}\|^2<\infty.
\end{eqnarray} 
Note that condition \eqref{con} involves  the  knowledge of the iterates $x_n$ and $x_{n-1}$ that are a priori unknown. However,  it can be 
ensured in practice by using the suitable on-line rule: $0\leq \theta_n \leq \bar{\theta}_n$, where
\begin{equation} \label{OR}
\bar{\theta}_n=
\begin{cases}
\min\bigg\{\theta, \dfrac{\epsilon_n}{\|x_n-x_{n-1}\|}\bigg\} \
\ \ \ \  &\text{\ if \ } x_n\neq x_{n-1} \\
\theta & \text{otherwise},
\end{cases}
\end{equation}
with $\sum_{n=1}^{\infty}\epsilon_n <\infty$ and $\theta \in [0,1)$. 
The on-line rule \eqref{OR} which also depends on the knowledge of the iterates $x_n$ and $x_{n-1}$, was considered in \cite{hur,vm} for solving EP \eqref{ep}.

\hfill

\noindent In \cite{bott} (see also \cite{Shehu}), the authors introduced the following condition on the inertial factor $\theta_n$:
 \begin{eqnarray*}
0=\theta_1\le \theta_n\le \theta_{n+1}\le \theta<1,~~\forall n\geq 1,
 \end{eqnarray*} 
\begin{eqnarray}\label{con2}
\tau>\frac{\theta^2(1+\theta)+\theta\sigma}{1-\theta^2},~~~0\le \phi\le \phi_n\le \frac{\tau-\theta[\theta(1+\theta)+\theta\tau+\sigma]}{\tau[1+\theta(1+\theta)+\theta\tau+\sigma]},
\end{eqnarray} where $\sigma,\tau>0.$  Unlike in \eqref{con} and \eqref{OR}, condition \eqref{con2} does not require any information on the iterates but on the coefficient $\phi_n$ and other parameters. However, it is complicated to get the upper bound of the inertial sequence even if $\phi_n$ is known. We can also see that the inertial factor is restrictive in \eqref{con2}.

\hfill
	
\noindent The main purpose of this paper is to consider an inertial factor with new conditions that only depend on the iteration coefficient $\phi_n,$ and where  the upper bound of the inertial sequence is easy to determine. Combining these relaxed inertial terms (i.e, the terms with $\theta_n$ and $\phi_n$) with  the subgradient extragradient method, we propose a new method for solving the EP \eqref{ep} when $F$ is pseudomonotone. We prove that the proposed method converges weakly to a solution of EP \eqref{ep}. Furthermore, we present some numerical experiments for our proposed method in comparison with other related methods in the literature.
	
\hfill	
	
\noindent The  rest of the paper is organized as  follows:  In Section \ref{Se2}  we recall  some  basic definitions and results  required for our convergence analysis. Section \ref{Se3} presents  and discusses the features of our proposed method. In Section \ref{Se4}, we study the convergence of this method. In Section \ref{Se6}, we carry out some  numerical experiments of our method in comparison with other  methods  in the literature. We conclude in Section \ref{Se7}.
\section{Preliminaries}\label{Se2}
\noindent In this section, we recall some  lemmas and definitions  which will be  needed  in the subsequent sections. \noindent Let  $\mathcal{H}$   be a real Hilbert space with inner product $\langle \cdot,  \cdot \rangle$, and associated norm $||\cdot||$ defined by $||x||=\sqrt{\langle x, x\rangle}, ~~\forall~ x\in \mathcal{H}$. We  denote  the weak convergence by  \quotes{$\rightharpoonup$}. 
	
\begin{definition}
 The domain of a function $F:\mathcal{H}\to \mathbb{R}\cup \{\infty\}$ is defined by $D(F)=\{x\in \mathcal{H}: F(x)<\infty\}$. The function  $F:D(F)\subseteq \mathcal{H}\to \mathbb{R}\cup \{\infty\}$ is said to be {\it lower semicontinuous  at a point $x\in D(F)$}, if
$$F(x)\leq \liminf_{x_n\to x} F(x_n).$$
\end{definition}
\begin{definition}
Let $F:\mathcal{H}\to (-\infty,\infty]$ be proper.  The subdifferential of $F$ at $x\in \mathcal{H}$ is 
\begin{align*}
\partial_2 F(x)=\left\{u\in\mathcal{H}|~F(y)\geq F(x)+\langle y-x,~u\rangle, ~\forall ~y\in\mathcal{H}\right\}.
\end{align*}
\end{definition}

\noindent  The normal cone $N_{\mathcal{C}}$ of $\mathcal{C}$ at $x\in\mathcal{C}$ is defined by 
\begin{align*}
N_{\mathcal{C}}(x)=\{w\in\mathcal{H}:\langle w, y-x\rangle \leq 0,~~\forall~ y\in\mathcal{C}\}.
\end{align*}

\begin{lemma}\cite{ppo}\label{lpp}
Let $\mathcal{C}$ be a nonempty closed and convex subset of  $\mathcal{H}$ and $g:\mathcal{H}\to \mathbb{R}\cup\{\infty\}$ be a proper, convex and lower semicontinuous functions on $\mathcal{H}.$ Assume either that $g$ is continuous at some point of $\mathcal{C},$ or that there is an interior point of $\mathcal{C}$ where $g$ is finite. Then, $\bar{x}$ is a solution to the following  convex problem $\min\{g({x}):x\in\mathcal{C}\}$ if and only if $0\in\partial_2 g(\bar{x})+N_{\mathcal{C}}(\bar{x}),$ where $\partial_2 g(\cdot)$ denotes the subdifferential of $g$ and $N_{\mathcal{C}}(\bar{x})$ is the normal cone of $\mathcal{C}$ at $\bar{x}.$
\end{lemma}
	
\begin{lemma}\label{lem1}\cite{bab} Let $\mathcal{H}$ be a real Hilbert space, then the following assertions hold:
\begin{enumerate}
\item [(1)] $2\langle x, y \rangle =\|x\|^2+\|y\|^2-\|x-y\|^2=\|x+y\|^2-\|x\|^2-\|y\|^2,~~\forall x,y \in \mathcal{H};$
\item [(2)] $\|\alpha x+(1-\alpha)y\|^2 = \alpha\|x\|^2+(1-\alpha)\|y\|^2-\alpha(1-\alpha)\|x-y\|^2,~~\forall x,y \in \mathcal{H},~ \alpha \in \mathbb{R}$.
\end{enumerate}
\end{lemma}

\begin{lemma}\cite{opial}\label{lem4}
Let $\mathcal{C}$ be a nonempty subset of $\mathcal{H}$ and let $\{x_n\}$ be a sequence in $\mathcal{H}$ such that the following two conditions hold:
\begin{itemize}
\item [(a)] for each $p\in\mathcal{C},~~\lim\limits_{n\to\infty}\|x_n-p\|$ exists;
			\item [(b)] every sequential weak cluster point of $\{x_n\}$ belongs to $\mathcal{C}.$
		\end{itemize}Then, $\{x_n\}$ converges weakly to a point in $\mathcal{C}.$
	\end{lemma}
\begin{lemma}\label{lem5}\cite{alv3}
Let $\{\gamma_n\}, \{\psi_n\}$ and $\{t_n\}$ be nonnegative sequences. Assume that 
\begin{align*}
\gamma_{n+1}\le \gamma_n+\psi_n(\gamma_n-\gamma_{n-1})+t_n,
\end{align*} and $0\le \psi_n\le \psi<1$ and $\sum\limits_{n=1}^{+\infty}t_n<+\infty.$ Then, $\lim\limits_{n\to+\infty}\gamma_n$ exists.
\end{lemma}
\section{proposed Method}\label{Se3}
\noindent In this section, we present our proposed method.  We begin by giving  the following assumptions under which our weak convergence result is obtained.
	
\begin{ass}\label{ass2}
	Let $F:\mathcal{C}\times\mathcal{C}\to\mathbb{R}$ be a function satisfying the following assumptions
	\begin{enumerate}
		\item [(1)] $F(x,x)=0, ~\forall~ x\in \mathcal {C};$
		\item[(2)]  $F$ is pseudomonotone on $\mathcal{C};$ 
		\item[(3)] $F$ satisfies the Lipschitz-like condition on $\mathcal{H}$ with constants $a_1$ and $a_2$;
		\item[(4)] $F(x,\cdot)$ is convex, lower semicontinuous and subdifferential on $\mathcal{C}$ for every $x\in\mathcal{C};$
		\item[(5)] $F(\cdot, y)$ is continuous on $\mathcal{C}$ for every $y\in\mathcal{C}.$
	\end{enumerate}
\end{ass}

\noindent\begin{ass}\label{gtm5as1} 
For all $n\ge 1$ and sufficiently small $\epsilon>0,$ let $0=\theta_1\le \theta_n\le \theta_{n+1}$ and:	
	\begin{itemize}
\item [(i)] $\theta_{n+1}\le \beta_n,~~$ if $\phi_n\in(0,0.5),~~~~\phi_n^{-1}-\phi_{n+1}^{-1}+3>0,$ where
\begin{align}\label{3.1}
\beta_n:=\frac{1}{2}\frac{1}{\phi_{n+1}^{-1}-2}\Big(\phi_n^{-1}+\phi_{n+1}^{-1}-1-\triangle_n\Big)
\end{align} with
\begin{align}\label{3.2}
\triangle_n:=\sqrt{(\phi_n^{-1}+\phi_{n+1}^{-1}-1)^2-4(\phi_n^{-1}-1-\epsilon)(\phi_{n+1}^{-1}-2)}.\\ \nonumber
\end{align}

\item[(ii)] $\theta_{n+1}\le\frac{1-\epsilon}{3},$ if $\phi_n\equiv0.5.$\\
\item[(iii)] $\theta_{n+1}\leq \sqrt{p_n^2+q_n}-p_n,$  if $\phi_n\in(0.5,1-\epsilon],$ where 
\begin{align}\label{3.3}
p_n:=\frac{1}{2}\frac{1}{2-\phi_{n+1}^{-1}}\Big(\phi_n^{-1}+\phi_{n+1}^{-1}-1\Big)\end{align} and 
\begin{align}\label{3.3b} q_n:=\frac{1}{2-\phi_{n+1}^{-1}}\Big(\phi_n^{-1}-1-\epsilon\Big).
\end{align}
\end{itemize}
\end{ass}

\hfill
		
\hrule
\begin{alg} \label{alg2}   Relaxed inertial subgradient extragradient method with adaptive stepsize strategy.
\hrule \hrule
\noindent {\bf{Step 0:}} Choose initial points $x_0,x_1\in\mathcal{H},$ let $\lambda_1>0, \mu\in(0,1)$ and set $n=1.$ 
		
\noindent 	{\bf Step 1}: Given the current iterates $x_{n-1}$ and $x_n~~ (n \geq 1),$ compute
\begin{align*}
w_n=x_n+\theta_n(x_n-x_{n-1})
		\end{align*}
		and
		$$y_n=	\arg\min\Big\{\lambda_n F(w_n,y)+\frac{1}{2}\|w_n-y\|^2:~ y\in \mathcal{C}\Big\}.$$
		If $y_n=w_n$: STOP. Otherwise, go to {\bf Step 2.}\\
		\noindent	{\bf Step 2}: Choose $\omega_n\in \partial_2 F(w_n,y_n)$ and $w^*\in N_{\mathcal{C}}(y_n)$  such that  $w^*=w_n-\lambda_n\omega_n-y_n$ and  construct the half-space
		\begin{align*}
			T_n=\{x\in \mathcal{H}:\langle w_n-\lambda_n \omega_n-y_n, \hspace{0.1cm}x-y_n \rangle\leq 0\}.
		\end{align*}
		
		\noindent	Then, compute
		$$z_n=\arg\min\Big\{\lambda_n F(y_n,y)+\frac{1}{2}\|w_n-y\|^2:~y\in T_n\Big\}.$$
		{\bf STEP 3:} Compute
		$$x_{n+1}=(1-\phi_n)w_n+\phi_nz_n,$$
		where
		
		\begin{eqnarray}\label{eq1}
			\lambda_{n+1}=\begin{cases}
				\min \left\{\frac{\mu\left(\|w_n-y_n\|^2+\|z_n-y_n\|^2\right)}{2(F(w_n,z_n)-F(w_n,y_n)-F(y_n,z_n))},~\lambda_n\right\}, & \mbox{if}~ F(w_n,z_n)-F(w_n,y_n)-F(y_n,z_n) >0,\\\\
				\lambda_n,& \mbox{otherwise}.
			\end{cases}
		\end{eqnarray}
		Set $n:=n+1$ and return to {\bf Step 1.}
		\hrule\hrule
	\end{alg}

\hfill

\begin{remark}\label{rmk}
	\begin{itemize}
\item[(a).] $\phi_{n+1}^{-1}-2>0$  in (i) since $\theta_{n+1}<\frac{1}{2}.$ Similarly,  $2-\phi_{n+1}^{-1}>0$ in (iii).\\

\item [(b).] The sequence $\{\triangle_n\}$  defined in \eqref{3.2} is well-defined. Indeed,
\begin{eqnarray*}
	&& (\phi_n^{-1}+\phi_{n+1}^{-1}-1)^2-4(\phi_{n+1}^{-1}-2)(\phi_n^{-1}-1-\epsilon)\\
	&=&\Big(\phi_n^{-1}-\phi_{n+1}^{-1}\Big)^2+6\phi_n^{-1}+2\phi_{n+1}^{-1}+4\Big(\phi_{n+1}^{-1}-2\Big)\epsilon-7\\
	&>&0. 
\end{eqnarray*} 
\end{itemize}
Therefore, Assumption \ref{gtm5as1} is valid.
\end{remark}

\hfill

\begin{remark}
In contrast  to the assumptions in \cite{mouda,hur,vm}, Assumption \ref{3.2} does not require the knowledge of the iterates. Also, unlike in \cite{bott}, the choice of $\theta_n$ is relaxed and its upper bound is easy to obtain; once $\phi_n$ is chosen, it becomes very easy to compute $\theta_n.$
\end{remark}
	
	\hfill

\begin{lemma}\cite{tta}\label{lem2}
The sequence $\{\lambda_n\}$ generated by Algorithm \ref{alg2} is a monotonically decreasing sequence with lower bound $\min\left\{\frac{\mu}{2\max\{a_1,a_2\},\lambda_1}\right\}.$
\end{lemma}

\begin{lemma}\label{llb}\cite{thong}
Let $\{z_n\}$ be a sequence generated by Algorithm  \ref{alg2} under Assumption \ref{ass2}. Then, for each $\bar{w}\in EP(F, \mathcal{C}),$ the following inequality holds:
\begin{align*}
\|z_n-\bar{w}\|^2\leq \|w_n-\bar{w}\|^2-\Big(1-\frac{\lambda_n\mu}{\lambda_{n+1}}\Big)\left[\|w_n-y_n\|^2+\|z_n-y_n\|^2\right].
		\end{align*}
	\end{lemma}
	
\section{Convergence Analysis}\label{Se4}
\begin{lemma}\label{gtm5bdn}
Let $\{x_n\}$ be a sequence generated by Algorithm  \ref{alg2} under Assumption \ref{ass2} and  Assumption \ref{gtm5as1}. Then,  for  $\bar{w}\in EP(F,\mathcal{C}),$
\begin{eqnarray*}
\Gamma_{n+1}\le \Gamma_n-\epsilon\|x_{n+1}-x_n\|^2,
\end{eqnarray*} where $\Gamma_n=\|x_n-\bar{w}\|^2-\theta_n\|x_{n-1}-\bar{w}\|^2+\delta_n\|x_n-x_{n-1}\|^2$ and  $\delta_n=(1+\theta_n)\theta_n+\phi_n^{-1}(1-\phi_n)(1-\theta_n)\theta_n.$ 
\end{lemma}
\begin{proof}
Let $\bar{w}\in EP(F,\mathcal{C}).$ From the definition of $w_n$ in Step 1 and Lemma \ref{lem1} (2), we have 
\begin{align}\label{aa1}\nonumber
\|w_n-\bar{w}\|^2&=\|x_n+\theta_n(x_n-x_{n-1})-\bar{w}\|^2\\ \nonumber
&=\|(1+\theta_n)(x_n-\bar{w})-\theta_n(x_{n-1}-\bar{w})\|^2\\ 
&=(1+\theta_n)\|x_n-\bar{w}\|^2-\theta_n\|x_{n-1}-\bar{w}\|^2+(1+\theta_n)\theta_n\|x_n-x_{n-1}\|^2.
\end{align}
Also, from the definition of $x_{n+1}$ and Lemma \ref{llb}, we have 
\begin{align*}\nonumber
\|x_{n+1}-\bar{w}\|^2&=\|(1-\phi_n)w_n+\phi_nz_n-\bar{w}\|^2\\ \nonumber
&=\|(1-\phi_n)(w_n-\bar{w})+\phi_n(z_n-\bar{w})\|^2\\ \nonumber
&=(1-\phi_n)\|w_n-\bar{w}\|^2+\phi_n\|z_n-\bar{w}\|^2-\phi_n(1-\phi_n)\|z_n-w_n\|^2\\ \nonumber
&\le(1-\phi_n)\|w_n-\bar{w}\|^2+\phi_n\|w_n-\bar{w}\|^2-\phi_n\left(1-\mu\frac{\lambda_n}{\lambda_{n+1}}\right)\left[\|w_n-y_n\|^2+\|z_n-y_n\|^2\right]\\ \nonumber&\;\;\;-\phi_n(1-\phi_n)\|z_n-w_n\|^2\\ 
&=\|w_n-\bar{w}\|^2-\phi_n\left(1-\mu\frac{\lambda_n}{\lambda_{n+1}}\right)\left[\|w_n-y_n\|^2+\|z_n-y_n\|^2\right]-\phi_n(1-\phi_n)\|z_n-w_n\|^2.
\end{align*}
\noindent By Lemma \ref{lem2}, we have 
\begin{align}\label{gtm5a5}
	\lim\limits_{n\rightarrow \infty}\Big(1-\frac{\lambda_n\mu}{\lambda_{n+1}}\Big)=1-\mu>0.
\end{align}Thus,
there exists $n_0\geq 1$ such that for all $n\geq n_0,$ $\Big(1-\frac{\lambda_n\mu}{\lambda_{n+1}}\Big)>0.$
Hence,
\begin{eqnarray}\label{aa2}
\|x_{n+1}-\bar{w}\|^2\le  \|w_n-\bar{w}\|^2-\phi_n(1-\phi_n)\|z_n-w_n\|^2.
\end{eqnarray}
\noindent From the definition of $x_{n+1}$, we have 
\begin{align*}
z_n-w_n={\phi_n}^{-1}(x_{n+1}-w_n).
\end{align*}Thus,
\begin{eqnarray}\label{aa3}
\|z_n-w_n\|^2&=&{\phi_n^{-2}}\|x_{n+1}-w_n\|^2\nonumber\\ 
&=&{\phi_n^{-2}}\|x_{n+1}-x_n-(x_n-x_{n-1})+(1-\theta_n)(x_n-x_{n-1})\|^2\nonumber\\
	&=& {\phi_n^{-2}}\|x_{n+1}-2x_n+x_{n-1}\|^2+\phi_n^{-2}(1-\theta_n)^2\|x_n-x_{n-1}\|^2\nonumber\\
	&&+2{\phi_n^{-2}}(1-\theta_n)\langle x_{n+1}-2x_n+x_{n-1}, x_n-x_{n-1}\rangle \nonumber\\
	&=& {\phi_n^{-2}}\|x_{n+1}-2x_n+x_{n-1}\|^2+\phi_n^{-2}(1-\theta_n)^2\|x_n-x_{n-1}\|^2\nonumber\\
	&&+{\phi_n^{-2}}(1-\theta_n)\left[\|x_{n+1}-x_{n}\|^2- \|x_{n}-x_{n-1}\|^2-\|x_{n+1}-2x_{n}+x_{n-1}\|^2\right]\nonumber\\
	&=& \phi_n^{-2}\theta_n \|x_{n+1}-2x_n+x_{n-1}\|^2+\phi_n^{-2}(1-\theta_n)^2\|x_n-x_{n-1}\|^2\nonumber\\
	&&+\phi_n^{-2}(1-\theta_n)\left[\|x_{n+1}-x_{n}\|^2- \|x_{n}-x_{n-1}\|^2\right]\nonumber\\ 
&\ge &\phi_n^{-2}(1-\theta_n)^2\|x_n-x_{n-1}\|^2+\phi_n^{-2}(1-\theta_n)\Big[\|x_{n+1}-x_n\|^2-\|x_n-x_{n-1}\|^2\Big].
\end{eqnarray} Substituting \eqref{aa1} and \eqref{aa3} into \eqref{aa2}, we have
\begin{align}\label{bb}\nonumber
\|x_{n+1}-\bar{w}\|^2&\le (1+\theta_n)\|x_n-\bar{w}\|^2-\theta_n\|x_{n-1}-\bar{w}\|^2+(1+\theta_n)\theta_n\|x_n-x_{n-1}\|^2\\\nonumber 
&\;\;\;\;-\phi_n(1-\phi_n)\Big(\phi_n^{-2}(1-\theta_n)^2\|x_n-x_{n-1}\|^2+\phi_n^{-2}(1-\theta_n)\Big[\|x_{n+1}-x_n\|^2-\|x_n-x_{n-1}\|^2\Big]\Big)\\ \nonumber
&=\|x_n-\bar{w}\|^2+\theta_n(\|x_n-\bar{w}\|^2-\|x_{n-1}-\bar{w}\|^2)-\phi_n^{-1}(1-\phi_n)(1-\theta_n)\|x_{n+1}-x_n\|^2\\
&\;\;\;\;+\delta_n\|x_n-x_{n-1}\|^2,
\end{align}where
$\delta_n=(1+\theta_n)\theta_n+\phi_n^{-1}(1-\phi_n)(1-\theta_n)\theta_n.$

\noindent  This implies that
\begin{eqnarray}\label{aa4}
&\|x_{n+1}-\bar{w}\|^2-\|x_n-\bar{w}\|^2-\theta_n(\|x_n-\bar{w}\|^2-\|x_{n-1}-\bar{w}\|^2)-\delta_n\|x_n-x_{n-1}\|^2+\delta_{n+1}\|x_{n+1}-x_{n}\|^2\\ \nonumber
&\le -\Big(\phi_n^{-1}(1-\phi_n)(1-\theta_n)-\delta_{n+1}\Big)\|x_{n+1}-x_n\|^2.
\end{eqnarray}
Using the fact that $\theta_n\le \theta_{n+1}$ and \eqref{aa4}, we have 
\begin{align*}\nonumber
-\Big(\phi_n^{-1}(1-\phi_n)(1-\theta_n)-\delta_{n+1}\Big)\|x_{n+1}-x_n\|^2
&\ge\|x_{n+1}-\bar{w}\|^2-\|x_n-\bar{w}\|^2-\theta_n(\|x_n-\bar{w}\|^2-\|x_{n-1}-\bar{w}\|^2)\\
&\;\;\;\;-\delta_n\|x_n-x_{n-1}\|^2+\delta_{n+1}\|x_{n+1}-x_n\|^2\\
&\ge \|x_{n+1}-\bar{w}\|^2-\|x_n-\bar{w}\|^2-\theta_{n+1}\|x_n-\bar{w}\|^2\\
&\;\;\;\;+\theta_n\|x_{n-1}-\bar{w}\|^2-\delta_n\|x_n-x_{n-1}\|^2+\delta_{n+1}\|x_{n+1}-x_n\|^2,
\end{align*}which implies that 
\begin{align}\label{3.4}
\Gamma_{n+1}\leq \Gamma_n-\Big(\phi_n^{-1}(1-\phi_n)(1-\theta_n)-\delta_{n+1}\Big)\|x_{n+1}-x_n\|^2,
\end{align}		where $\Gamma_n=\|x_n-\bar{w}\|^2-\theta_n\|x_{n-1}-\bar{w}\|^2+\delta_n\|x_n-x_{n-1}\|^2.$ 

\noindent Now, observe for $\epsilon>0,$ we have 
\begin{align}\label{aa6}\nonumber
\phi_n^{-1}(1-\phi_n)(1-\theta_n)-\delta_{n+1}-\epsilon&=\phi_n^{-1}(1-\phi_n)(1-\theta_n)-\theta_{n+1}-\theta_{n+1}^2\\ \nonumber&\;\;\;\;+\phi_{n+1}^{-1}(1-\phi_{n+1})(\theta_{n+1}-1)\theta_{n+1}-\epsilon\\ \nonumber
&\ge\phi_n^{-1}(1-\phi_n)(1-\theta_{n+1})-\theta_{n+1}-\theta_{n+1}^2\\ \nonumber&\;\;\;\;+\phi_{n+1}^{-1}(1-\phi_{n+1})(\theta_{n+1}-1)\theta_{n+1}-\epsilon\\
&=-\Big(2-\phi_{n+1}^{-1}\Big)\theta_{n+1}^2-\Big(\phi_n^{-1}+\phi_{n+1}^{-1}-1\Big)\theta_{n+1}+\phi_n^{-1}-1-\epsilon.
\end{align} Now, we consider three cases.
		
\noindent {\bf Case 1:} Suppose $\phi_n\in (0, 0.5).$ Then, from the condition $\theta_{n+1}\le \frac{1}{2}\frac{1}{\phi_{n+1}^{-1}-2}\Big(\phi_n^{-1}+\phi_{n+1}^{-1}-1-\triangle_n\Big)$ in \eqref{3.1}, we get
\begin{align*}
\triangle_n\le \phi_n^{-1}+\phi_{n+1}^{-1}-1-2(\phi_{n+1}^{-1}-2)\theta_{n+1}.
\end{align*} By Remark \ref{rmk} (b), we have that $\triangle_n\ge0.$ Hence,
\begin{align*}
\triangle_n^2&\le \Big[\Big(\phi_n^{-1}+\phi_{n+1}^{-1}-1\Big)-\Big(2(\phi_{n+1}^{-1}-2)\theta_{n+1}\Big)\Big]^2\\
&=(\phi_n^{-1}+\phi_{n+1}^{-1}-1)^2-4(\phi_n^{-1}+\phi_{n+1}^{-1}-1)(\phi_{n+1}^{-1}-2)\theta_{n+1}+4(\phi_{n+1}^{-1}-2)^2\theta_{n+1}^2.
\end{align*}That is,
\begin{eqnarray*}
&(\phi_n^{-1}+\phi_{n+1}^{-1}-1)^2-4(\phi_n^{-1}-1-\epsilon)(\phi_{n+1}^{-1}-2)\\&\le (\phi_n^{-1}+\phi_{n+1}^{-1}-1)^2-4(\phi_n^{-1}+\phi_{n+1}^{-1}-1)(\phi_{n+1}^{-1}-2)\theta_{n+1}+4(\phi_{n+1}^{-1}-2)^2\theta_{n+1}^2.
\end{eqnarray*}Hence,
\begin{eqnarray*}
&-\Big(\phi_n^{-1}+\phi_{n+1}^{-1}-1\Big)(\phi_{n+1}^{-1}-2)\theta_{n+1}+(\phi_{n+1}^{-1}-2)^2\theta_{n+1}^2+(\phi_n^{-1}-1-\epsilon)(\phi_{n+1}^{-1}-2)\geq 0.
\end{eqnarray*}Since $\ \phi_{n+1}^{-1}-2>0,$ we obtain
\begin{eqnarray*}
	&-\Big(\phi_n^{-1}+\phi_{n+1}^{-1}-1\Big)\theta_{n+1}+(\phi_{n+1}^{-1}-2)\theta_{n+1}^2+(\phi_n^{-1}-1-\epsilon)\geq 0.
\end{eqnarray*}That is,
\begin{align}\label{aa7}
-(2-\phi_{n+1}^{-1})\theta_{n+1}^2-\Big(\phi_n^{-1}+\phi_{n+1}^{-1}-1\Big)\theta_{n+1}+\phi_n^{-1}-1-\epsilon\geq 0.
\end{align}Using \eqref{aa7} in \eqref{aa6}, we get
\begin{eqnarray*}
&\phi_n^{-1}(1-\phi_n)(1-\theta_n)-\delta_{n+1}-\epsilon\geq -(2-\phi_{n+1}^{-1})\theta_{n+1}^2-\Big(\phi_n^{-1}+\phi_{n+1}^{-1}-1\Big)\theta_{n+1}+\phi_n^{-1}-1-\epsilon\geq 0,
\end{eqnarray*}which implies
\begin{align*}
-\Big(\phi_n^{-1}(1-\phi_n)(1-\theta_n)-\delta_{n+1}\Big)\le -\epsilon.
		\end{align*}
{\bf Case 2:} Suppose $\phi_n\equiv0.5.$ Then from \eqref{aa6}, we obtain
\begin{align*}
\phi_n^{-1}(1-\phi_n)(1-\theta_n)-\delta_{n+1}-\epsilon&\ge -3\theta_{n+1}+1-\epsilon\\
	&\ge 0,
\end{align*}since by Assumption \ref{gtm5as1} (ii), $\theta_{n+1}\le \frac{1-\epsilon}{3}.$ Hence, 
 \begin{align*}
 -\Big(\phi_n^{-1}(1-\phi_n)(1-\theta_n)-\delta_{n+1}\Big)\le -\epsilon.
\end{align*}
{\bf Case 3:} Suppose $\phi_n\in(0.5,1-\epsilon].$ Then from the condition $\theta_{n+1}\leq \sqrt{p_n^2+q_n}-p_n$,  we have
\begin{align*}
\Big(\theta_{n+1}+p_n\Big)^2\leq p_n^2+q_n,
\end{align*}which implies  that 
\begin{align*}
\theta_{n+1}^2+2p_n\theta_{n+1}-q_n\le 0.
\end{align*}Now, using \eqref{3.3} and \eqref{3.3b}, we get
\begin{eqnarray*}
&\theta_{n+1}^2+\frac{1}{2-\phi_{n+1}^{-1}}\Big(\phi_n^{-1}+\phi_{n+1}^{-1}-1\Big)\theta_{n+1}-\frac{1}{2-\phi_{n+1}^{-1}}(\phi_n^{-1}-1-\epsilon)\le 0.
\end{eqnarray*}Hence,
\begin{align*}
(2-\phi_{n+1}^{-1})\theta_{n+1}^2+(\phi_n^{-1}+\phi_{n+1}^{-1}-1)\theta_{n+1}-(\phi_n^{-1}-1-\epsilon)\le 0,
\end{align*} which implies 
\begin{align*}
-(2-\phi_{n+1}^{-1})\theta_{n+1}^2-(\phi_n^{-1}-\phi_{n+1}^{-1}-1)\theta_{n+1}+(\phi_n^{-1}-1-\epsilon)\ge 0.
\end{align*} Thus, by \eqref{aa6}, we obtain
\begin{align*}
	-\Big(\phi_n^{-1}(1-\phi_n)(1-\theta_n)-\delta_{n+1}\Big)\le -\epsilon.
\end{align*} Therefore, in all cases, we have established that
\begin{align*}
	-\Big(\phi_n^{-1}(1-\phi_n)(1-\theta_n)-\delta_{n+1}\Big)\le -\epsilon.
\end{align*}
Now, using this and  \eqref{3.4}, we get
\begin{align*}
\Gamma_{n+1}\le \Gamma_n-\epsilon\|x_{n+1}-x_n\|^2.
\end{align*}
\end{proof}

\begin{lemma}\label{lem3}
Let $\{x_n\}$ be generated  by Algorithm \ref{alg2} under Assumption \ref{ass2} and Assumption \ref{gtm5as1}. Then $\lim\limits_{n\to\infty} \|x_n-\bar{w}\|$ exists for all $\bar{w} \in EP(F, \mathcal C)$.
\end{lemma}
\begin{proof}
By Lemma \ref{gtm5bdn}, we see that $\{\Gamma_n\}$ is nonincreasing.
 Now, let $\gamma_n:=\|x_n-\bar{w}\|^2.$ Then, $\Gamma_n:=\gamma_n-\theta_n\gamma_{n-1}+\delta_n\|x_n-x_{n-1}\|^2.$ Thus,
\begin{align*}
\gamma_n-\theta_n\gamma_{n-1}\le \Gamma_n\le \Gamma_1,
\end{align*}which implies that
\begin{align*}
\gamma_n&\le \theta_n\gamma_{n-1}+\Gamma_1\\
&\le \theta\gamma_{n-1}+\Gamma_1\hspace{0.5cm}\text{(for some} ~~~~\theta<1, \text{since}~~~~~ \theta_n<1)\\
&\le \theta(\theta\gamma_{n-2}+\Gamma_1)+\Gamma_1\\
&=\theta^2\gamma_{n-2}+\theta\Gamma_1+\Gamma_1\\
&\vdots\\
&\leq \theta^n\gamma_1+\theta^{n-1}\Gamma_1+\cdots+\theta\Gamma_1+\Gamma_1\le \theta^n\gamma_1+\frac{\Gamma_1}{1-\theta}.
\end{align*}
Hence, for $j\le n-1,$ we obtain from Lemma \ref{gtm5bdn} that 
\begin{align*}
	\epsilon\sum\limits_{j=1}^{n-1}\|x_{j+1}-x_j\|^2\leq \Gamma_1-\Gamma_n\le \Gamma_1+\theta_n\gamma_{n-1}\le \Gamma_1+\theta^{n-1}\gamma_1+\frac{\Gamma_1}{1-\theta}.
\end{align*}
Thus we have that $\sum\limits_{j=1}^{n-1}\|x_{j+1}-x_j\|^2$ is  bounded for all $n.$ Hence, \begin{align}\label{bbd}\sum\limits_{n=1}^{+\infty}\|x_{n+1}-x_n\|^2<+\infty.\end{align}\noindent Now, from \eqref{bb} we obtain
\begin{align*}
\|x_{n+1}-\bar{w}\|^2\le \|x_n-\bar{w}\|^2+\theta_n\Big(\|x_n-\bar{w}\|^2-\|x_{n-1}-\bar{w}\|^2\Big)+\delta_n\|x_n-x_{n-1}\|^2.
\end{align*} From the previous inequality, \eqref{bbd}  and Lemma \ref{lem5}, we have that $\lim\limits_{n\to+\infty}\|x_n-\bar{w}\|$ exists. This implies that $\{x_n\}$ is bounded.
\end{proof}

\hfill

\begin{theorem}
	Let $\{x_n\}$  be generated  by Algorithm \ref{alg2} under Assumption \ref{ass2} and Assumption \ref{gtm5as1}. Then $\{x_n\}$ converges weakly to  $\bar{w}\in EP(F,\mathcal{C}).$
\end{theorem}
\begin{proof}
 \noindent From \eqref{bbd}, we have 
\begin{align}\label{aaf}
\lim\limits_{n \rightarrow \infty}\|x_{n+1}-x_n\|=0.
\end{align}
Also,  we have
\begin{align}\label{mnm2}\nonumber
\|x_n-w_n\|&=\|x_n-(x_n+\theta_n(x_n-x_{n-1}))\|\\
&\le \|x_n-x_{n-1}\|\to0,~~~\mbox{as}~~~n\to \infty.
\end{align}From \eqref{aaf} and \eqref{mnm2}, we get
\begin{eqnarray*}
\lim\limits_{n \rightarrow \infty}\|x_{n+1}-w_n\|\to 0, ~~~\mbox{as}~~~n\to\infty.
\end{eqnarray*}Hence,
\begin{eqnarray}\label{mnm3}
\lim\limits_{n \rightarrow \infty}\|z_n-w_n\|=\lim\limits_{n \rightarrow \infty}\phi^{-1}_n\|x_{n+1}-w_n\|=0.
\end{eqnarray}
From Lemma \ref{llb}, we get
\begin{align*}\nonumber
\Big(1-\mu\frac{\lambda_n}{\lambda_{n+1}}\Big)\|w_n-y_n\|^2&\le \|w_n-\bar{w}\|^2-\|z_n-\bar{w}\|^2\\ \nonumber
&\le \Big(\|w_n-\bar{w}\|+\|z_n-\bar{w}\|\Big)\Big(\|w_n-\bar{w}\|-\|z_n-\bar{w}\|\Big)\\
&\le \Big(\|w_n-\bar{w}\|+\|z_n-\bar{w}\|\Big)\|z_n-w_n\|\to 0, ~~~\mbox{as}~~~~n\to\infty.
\end{align*} Hence, 
\begin{eqnarray}\label{af1}
\lim\limits_{n \rightarrow \infty}\|w_n-y_n\|=0.
\end{eqnarray}
Furthermore, we have 
\begin{align*}
0\le \|x_n-y_n\|\le \|x_n-w_n\|+\|w_n-y_n\|\to 0,~~~\mbox{as}~~~n\to \infty.
\end{align*}From \eqref{aaf} to \eqref{af1}  we have 
\begin{align}\label{af2}
\|x_{n+1}-y_n\|\to 0, ~~~\mbox{as}~~~n\to\infty,~~ \|z_n-y_n\|\to 0, ~~~\mbox{as}~~~n\to\infty.
\end{align}
Next, we show that the set of all sequentially weak limit point of the sequence $\{x_n\}$ belongs to $EP(F,\mathcal{C}).$ Since $\{x_n\}$ is bounded we have that $\{x_n\}$ has at least one accumulation point, say $z\in\mathcal{H}.$ Assume that $\{x_{n_k}\}\subset\{x_n\}$ such that $x_{n_k}\rightharpoonup z, k\to\infty.$ Since $\|y_n-x_n\|\to 0,~n\to\infty,$ we have that $y_{n_k}\rightharpoonup z,~k\to\infty$ for some $\{y_{n_k}\}\subset \{y_n\}$.
		
\noindent Now,  from \cite[Lemma 3.2,  equation (14)]{thong} and \cite[Lemma 3.2, equation (6)]{thong}, we get 
\begin{align}\label{p11}
2F(y_n,z_n)\geq	\frac{2}{\lambda_n}\langle w_n-y_n,z_n-y_n\rangle -\frac{\mu}{\lambda_{n+1}}\left[\|w_n-y_n\|^2+\|z_n-y_n\|^2\right]
\end{align}
 and 
\begin{align}\label{p12}
\lambda_nF(y_n,y)\geq \lambda_n F(y_n, z_n)+\langle w_n-z_n,y-z_n\rangle,~~\forall~y\in\mathcal{C},
\end{align}
respectively. Combining \eqref{p11} and \eqref{p12}, we obtain
\begin{align}\label{p14}\nonumber
\lambda_{n_k}F(y_{n_k},y)&\geq \langle w_{n_k}-y_{n_k},z_{n_{k}}-y_{n_k}\rangle-\frac{\mu}{2}\frac{\lambda_{n_k}}{\lambda_{n_{k+1}}}\left[\|w_{n_k}-y_{n_k}\|^2+\|z_{n_{k}}-y_{n_k}\|^2\right]\\&+\langle w_{n_k}-z_{n_{k}},y-z_{n_{k}}\rangle,~\forall~y\in\mathcal{C}.
\end{align} Taking the limit in \eqref{p14} as $k\to \infty$ (taking note of \eqref{mnm3}-\eqref{af2}), we obtain
\begin{align*}
F(z,y)\geq 0,~\forall~y\in\mathcal{C}.
\end{align*}which implies that $z\in EP(F,\mathcal{C})$. Using this and Lemma \ref{lem3} in  Lemma \ref{lem4}, we have that $\{x_n\}$ converges weakly to an element in $EP(F,\mathcal{C}).$ This completes the proof.
\end{proof}

\hfill

\section{Numerical experiments}\label{Se6}
\noindent \noindent The focus of this section is to provide some computational experiments to demonstrate the effectiveness, accuracy and easy-to-implement nature of our proposed algorithms. We compare our proposed algorithm (Algorithm \ref{alg2}) with  Algorithm 1 in \cite{hur} and Algorithm 1 in \cite{vm}. Throughout this section, we shall name these algorithms RKSPW (Alg. 1) and VM (Alg. 1), respectively.

\hfill	
	
\begin{example}\label{EX1}
We consider the Nash-Cournot oligopolistic equilibrium model in \cite{FP} where the bifunction $F$ in $\mathbb{R}^N$ is of the form:
		\begin{align*}
			F(x,y)=\langle Px+Qy+q,y-x\rangle,
		\end{align*}where $q\in\mathbb{R}^N$ and $P,Q$ are two matrices of order $N$ such that $Q$ is symmetric positive semi-definite and $Q-P$ is symmetric negative semi-definite. The feasible set is defined as $
			\mathcal{C}=\left\{x\in\mathbb{R}^N:-5\leq x_i\leq 5\right\}$. The bifunction satisfies Assumption \ref{ass2} with $a_1=a_2=\frac{1}{2}\|P-Q\|$ (see \cite{qu}). The vectors $x_0, x_1, q$ are generated randomly and uniformly  in $[-N,N]$ and the two matrices $P,Q$ are generated randomly such that their properties are satisfied. 
\end{example}

\hfill

\begin{example}\label{EX3}
Let $\mathcal{H}=\left(\ell_2(\mathbb{R}),\|\cdot\|_2\right)$ be the linear spaces whose elements are all 2-summable sequences $\{x_i\}^{\infty}_{i=1}$ of scalars in $\mathbb{R},$ that is
$$\ell_2(\mathbb{R})=\mathcal{H}=\Biggl\{x=(x_1,x_2,\cdots,x_i,\cdots), x_i\in\mathbb{R}:\sum\limits_{i=1}^{\infty}|x_i|^2<\infty\Biggl\}$$  with inner product $\langle \cdot,\cdot\rangle:\ell_2\times\ell_2\to\mathbb{R}$ and norm $\|\cdot\|:\ell_2\to\mathbb{R}$ defined by $\langle x,y\rangle :=\sum\limits_{j=1}^{\infty}x_iy_i$ and  $\|x\|_2=\left(\sum\limits_{i=1}^{\infty}|x_i|^2\right)^{\frac{1}{2}}$, for $x=\{x_i\}^{\infty}_{i=1},~y=\{y_i\}^{\infty}_{i=1} \in \ell_2(\mathbb{R}).$  Let $\mathcal{C}=\Biggl\{x\in\mathcal{H}:\|x\|\leq 1\Biggl\}.$ Define the bifunction	$F:\mathcal{C}\times\mathcal{C}\to\mathbb{R}$ by $F(x,y)=(3-\|x\|)\langle x,y-x\rangle, \forall~x,y\in\mathcal{C}.$ It is easy to show that $F$ is a pseudomonotone bifunction which is not monotone and $F$ satisfies the Lipschitz-type condition with constants $a_1=a_2=\frac{5}{2}.$ Also, $F$ satisfies  Assumptions \ref{ass2} ((4)-(5)).  We consider the following cases for the numerical experiments of this  example

\noindent {\bf Case 1:}
Take  $x_1= \Biggl(\frac{5}{7},\frac{1}{7},\frac{1}{35},\cdots\Biggl)$ and $x_0=\Biggl(\frac{1}{2},\frac{1}{6},\frac{1}{18},\cdots\Biggl)$.

\noindent {\bf Case 2:}
Take  $x_1= \Biggl(\frac{1}{2},\frac{1}{6},\frac{1}{18},\cdots\Biggl)$ and $x_0=\Biggl(\frac{1}{3},\frac{1}{9},\frac{1}{27},\cdots\Biggl)$.
		
\noindent {\bf Case 3:}
Take  $x_1= \Biggl(\frac{1}{3},\frac{1}{9},\frac{1}{27},\cdots\Biggl)$ and $x_0=\Biggl(\frac{2}{5},\frac{1}{5},\frac{1}{10},\cdots\Biggl)$.
	\end{example}
	\noindent During the computation, we make use of the following:
\begin{itemize}
        \item Algorithm \ref{alg2}:  $\lambda_1 = 0.1, \mu=0.5, \epsilon=0.000001$, $\theta_n=\beta_n$ (when $\phi_n=\frac{n-0.5}{2n}$), where $\beta_n$ is as defined in \eqref{3.1}, $\theta_n=\frac{1-\epsilon}{3}$ (when $\phi_n=0.5)$ and $\theta_n=\sqrt{p_n^2+q_n}-p_n$ (when $\phi_n=\frac{n-0.1}{n}$), where $p_n$ and $q_n$ are as defined in \eqref{3.3} and \eqref{3.3b}, respectively.

\item RKSPW (Alg. 1) in \cite{hur}:  $\lambda_1 = 0.1, \delta=0.9, \sigma=0.9 \min\{1,0.5 a_1, 0.5a_2\}, \mu=0.9 \sigma$ and $\epsilon_n=\frac{100}{(n+1)^2}$.

\item VM (Alg. 1) in \cite{vm}: $\lambda=0.9 \min\{0.5 a_1, 0.5a_2\}, \theta=0.9$ and $\epsilon_n=\frac{100}{(n+1)^2}$.
\end{itemize}
\noindent  We then use the stopping criterion;  $\mbox{TOL}_n:=\|y_n-w_n\|<\varepsilon$ for all algorithms, where $\varepsilon$ is the predetermined error. 

\noindent All the computations are performed using Matlab 2016 (b) which is running on a  personal computer with an Intel(R) Core(TM) i5-2600 CPU at 2.30GHz and 8.00 Gb-RAM.\\

\noindent In the tables below, \quotes{Iter} means the number of iterations.  Also, in the tables and figures, Alg. 3.3, RKSPW (Alg. 1), and VM (Alg. 1) represent Algorithm \ref{alg2}, Algorithm 1 in \cite{hur}  and Algorithm 1 in \cite{vm}, respectively.

	\begin{center}
{\bf Table 1. Numerical results for Example \ref{EX1} with $\epsilon=10^{-5}$.}
		\adjustbox{width=15cm}{ 
			
\begin{tabular}{c c c c c c c c c c c c c c c c c c}
		\hline\\
		& \multicolumn{2}{c}{N=20} & \multicolumn{2}{c}{N=50} & \multicolumn{2}{c}{N=100} \\
		\cline{2-3}\cline{4-5}\cline{6-7}\cline{8-9}\cline{10-11}\cline{12-13}\\
	\hline\\
	Algorithms	&  CPU Time & Iter. & CPU Time & Iter. & CPU Time& Iter. \\
	\hline\\
	\noindent	Alg. \ref{alg2} $\big(\phi_n=\frac{n-0.5}{2n}\big)$&  5.2040 & 152 & 9.7123 & 233 & 10.3026 & 264\\ [0.5ex]
		\hline \\
		  \noindent Alg. \ref{alg2} $\big(\phi_n=0.5\big)$  &  5.1138 & 122 & 7.5978 & 176 & 7.9771 & 186 \\ [0.5ex]
		 \hline\\
Alg. \ref{alg2} $\big(\phi_n=\frac{n-0.1}{n}\big)$&  3.9750 & 82 & 5.6901 & 119& 6.4106 & 135\\ [0.5ex]
\hline \\
 RKSPW (Alg. 1) in \cite{hur} &  17.1610& 361 & 15.7162& 339 & 13.9429 & 304\\ [0.5ex]
\hline \\
 VM (Alg. 1) in \cite{vm} &11.2082 &231 & 14.7417 & 305 & 13.5661 & 283\\ [0.5ex]
		 \hline \\
	\end{tabular}}
	\label{T1}
	\end{center}

\begin{figure}
	\begin{center}
		\includegraphics[width=7cm]{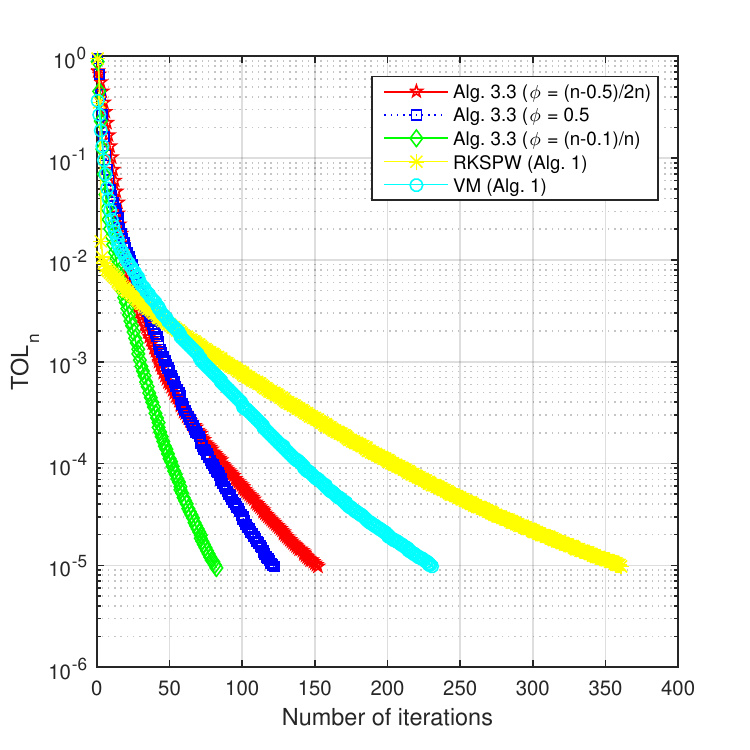}%
		\includegraphics[width=7cm]{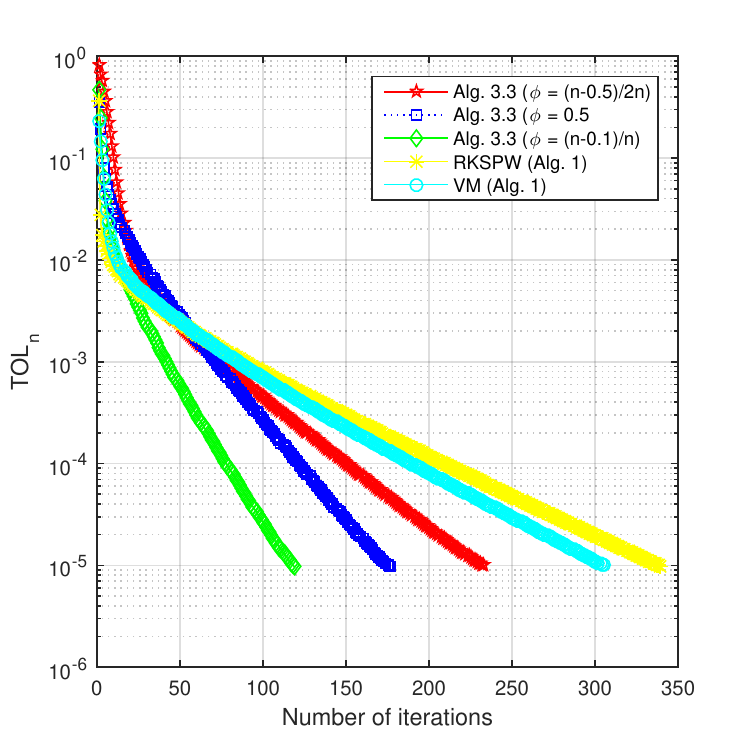}\\
		\includegraphics[width=7cm]{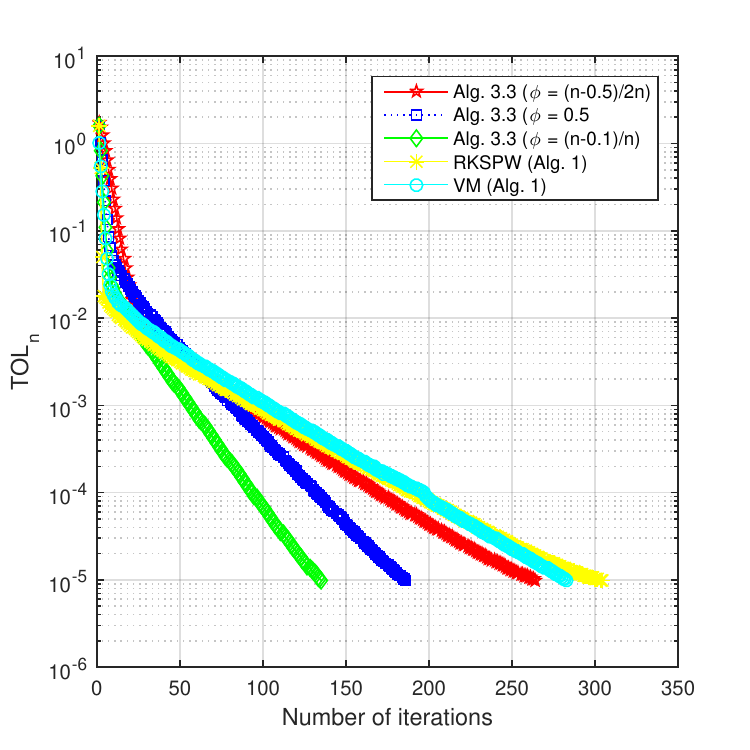}%
	\end{center}
	\caption{The behavior of $\mbox{TOL}_n$ with $\epsilon=10^{-5}$:
		Top Left: {\bf $N=20$}; Top Right: {\bf $N=50$}; Bottom: {\bf $N=100$}.} 
\end{figure}

\begin{center}
	{\bf Table 2. Numerical results for Example \ref{EX3} $\epsilon=10^{-5}$.}
\adjustbox{width=15cm}{ 

\begin{tabular}{c c c c c c c c c c c c c c c c c c}
		\hline\\
& \multicolumn{2}{c}{Case 1} & \multicolumn{2}{c}{Case 2} & \multicolumn{2}{c}{Case 3} \\
\cline{2-3}\cline{4-5}\cline{6-7}\cline{8-9}\cline{10-11}\cline{12-13}\\
\hline\\
Algorithms&  CPU Time & Iter. & CPU Time & Iter. & CPU Time& Iter. \\
\hline\\
\noindent	Alg. \ref{alg2} $\big(\phi_n=\frac{n-0.5}{2n}\big)$& 0.0157 & 99 & 0.0114 & 106 & 0.0107 & 102\\ [0.5ex]
\hline \\
\noindent Alg. \ref{alg2} $\big(\phi_n=0.5\big)$  &  0.0090 & 63 & 0.0085 & 68 & 0.0062 & 66 \\ [0.5ex]
\hline\\
Alg. \ref{alg2} $\big(\phi_n=\frac{n-0.1}{n}\big)$&  0.0087 & 47 & 0.0087 & 51& 0.0083 & 49\\ [0.5ex]
\hline \\
RKSPW (Alg. 1 in \cite{hur}) &  0.1254& 110 & 0.1123& 110 & 0.1382 & 110\\ [0.5ex]
\hline \\
VM (Alg. 1 in \cite{vm}) &0.1391 &140& 0.1169 & 162 & 0.1471 & 150\\ [0.5ex]
\hline \\
\end{tabular}}
\label{T2}
\end{center}

\begin{figure}
\begin{center}
\includegraphics[width=7cm]{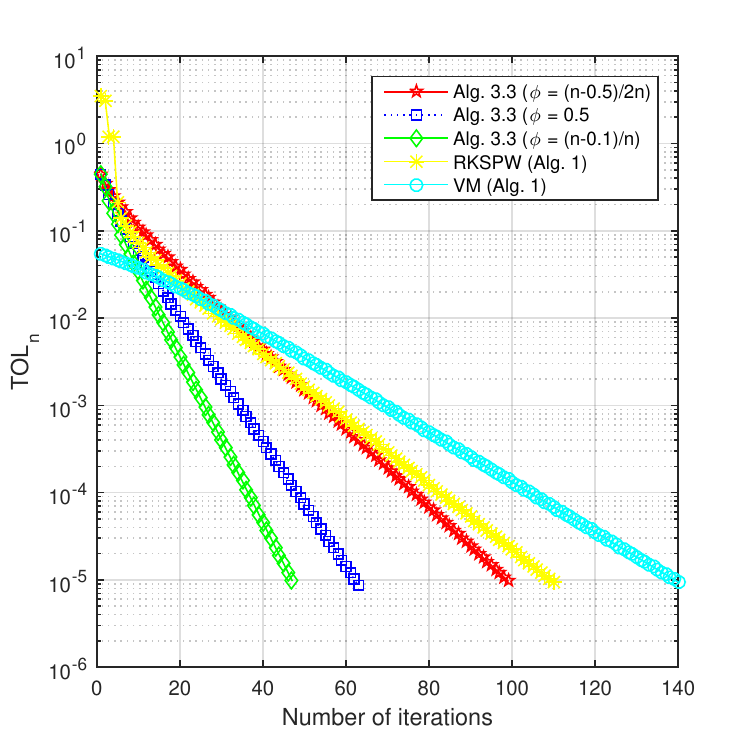}%
\includegraphics[width=7cm]{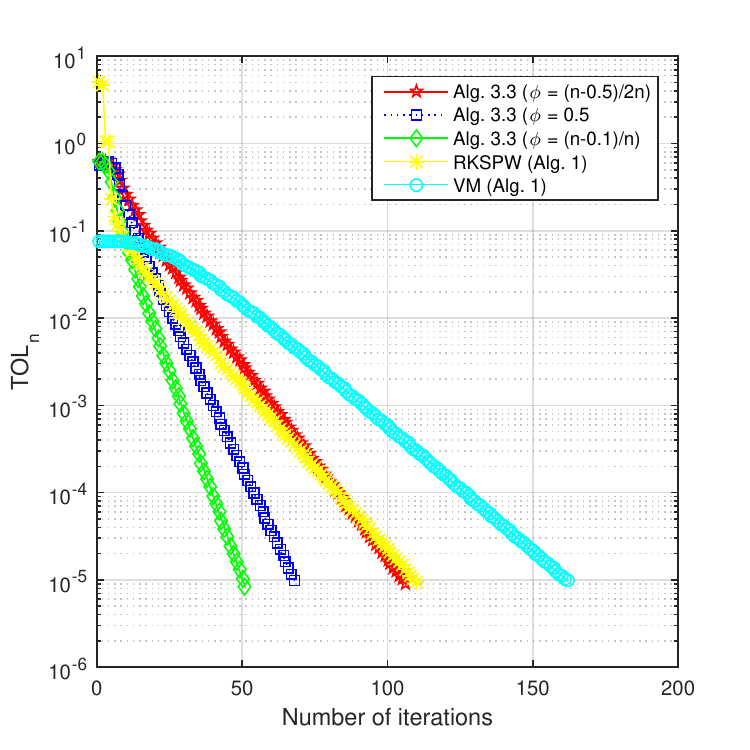}\\
\includegraphics[width=7cm]{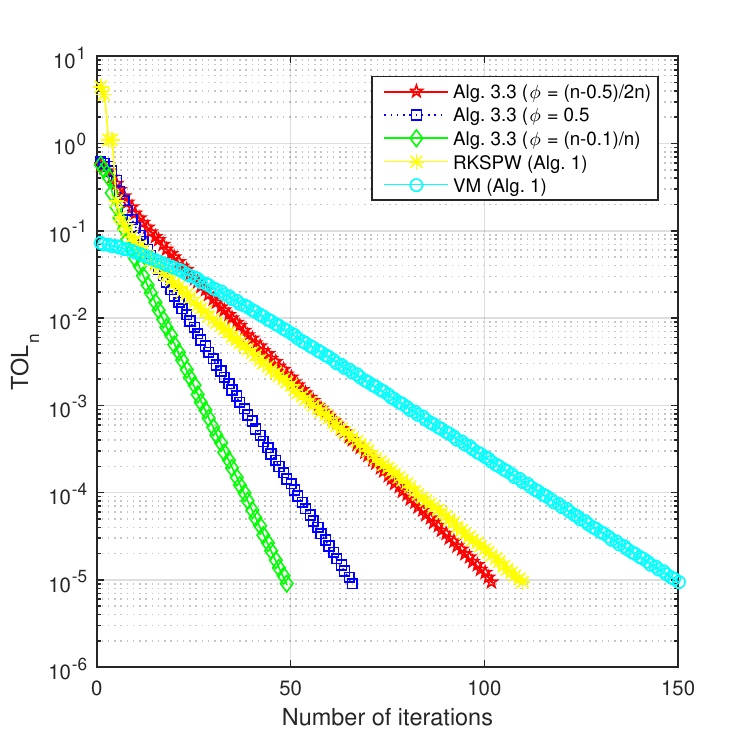}%
\end{center}
\caption{The behavior of $\mbox{TOL}_n$ with $\epsilon=10^{-5}$:
Top Left: {\bf Case 1}; Top Right: {\bf Case 2}; Bottom: {\bf Case 3}.} 
\end{figure}

	\section{Conclusion}\label{Se7}
\noindent We have considered in this paper, a new  inertial condition for the  subgradient extragradient method with self-adaptive step size for solving pseudomonotone equilibrium problem in a real Hilbert space.  It was proved that the sequence of iterates generated by our proposed method converges weakly  to a solution of the equilibrium problem under improved conditions on the inertial factor than many existing conditions in the literature.  Numerical results are given to support our analysis. 

	\section{acknowledgment}
	\noindent The  authors acknowledges with thanks the  School of Mathematics at University of the Witwatersrand for making their facilities available for the research.
	\hfill
	
	\noindent{\bf Funding}
	
	\noindent The second author is supported by the postdoctoral research grant from University of the Witwatersrand, South Africa.

	\noindent {\bf {Availability of data and material}}\\
	Not applicable.\\
	
	\noindent {\bf {Competing interests}}\\
	The authors declare that they have no competing interests.\\
	
	\noindent {\bf Authors' contributions}
	
	\noindent All authors worked equally on the results and approved the final manuscript.


\begin{thebibliography}{}


\bibitem{alv3}{\rm F Alvarez, H. Attouch,} An inertial proximal moethid for maximal monotone operators via discretization of a nonlinear oscillator with damping.,{\it Set-Valued Var. Anal.,} {\bf 9} (2001), 3-11.

\bibitem{att}{\rm  H. Attouch, J. Peypouquet,  P. Redont,}  A dynamical approach to an inertial forward-backward algorithm for convex minimization, {\it  SIAM J. Optim.,} {\bf 24} (1) (2014), 232-256.
\bibitem{bab}{\rm H. H. Bauschke, P. L. Combettes,} Convex Analysis and monotone operator theory in Hilbert spaces, Springer, New York (2011).
\bibitem {bot2} {\rm A. Beck,  M. Teboulle,} A fast iterative shrinkage-thresholding algorithm for linear inverse problems, {\it SIAM J. Imaging Sci.,} {\bf 2} (1) (2009), 183-202.
		
\bibitem{blum}{\rm E. Blum, W. Oettli,} From optimization and variational inequalities to equilibrium problems, {\it Math. Stud.,} {\bf 63} (1994), 123–145.
\bibitem{bott}{\rm R. I.  Boţ, E. R. Csetnek, C. Hendrich,} Inertial Douglas–Rachford splitting for monotone inclusion problems. Applied Mathematics and Computation, 256 (2015), 472-487.

\bibitem{Dong}{\rm Y. Dong,} New inertial factors of the Krasnosel'ski\u{\i}-Mann iteration, {\it Set-Valued Var. Anal.}, {\bf 29} (2021), 145-161.

\bibitem{FP} {\rm F.	Facchinei and J. S.   Pang,} Finite-dimensional variational inequalities and complementarity problems, Springer Science and  Business Media, (2007).



\bibitem {Reich} {Censor, Y.,  Gibali, A.,   Reich, S.:} The subgradient extragradient method for solving variational inequalities in Hilbert space. J. Optim. Theory Appl. {\bf 148}, 318-335  (2011)


		
\bibitem{he1}{\rm  D. V. Hieu,} Halpern subgradient extragradient method extended to equilibrium problems, {\it Rev. R. Acad. Cienc. Exactas Fís. Nat., Ser. A Mat.,} {\bf 111} (2017), 823-840.
\bibitem{he2}{\rm  D. V. Hieu,} Hybrid projection methods for equilibrium problems with non-Lipschitz type bifunctions, {\it Math. Methods Appl. Sci.,} {\bf 40} (2017), 4065-4079.
\bibitem{ii1}{\rm H. Iiduka,} A new iterative algorithm for the variational inequality problem over the fixed point set of a firmly nonexpansive mapping, {\it Optimization,} {\bf59} (2010), 873–885.
\bibitem{ii2}{\rm  H. Iiduka, I. Yamada,}  A use of conjugate gradient direction for the convex optimization problem over the fixed point set of a nonexpansive mapping, {\it SIAM J. Optim.,} {\bf 19} (2009), 1881-1893.


\bibitem{PA2} {\rm C. Izuchukwu, S. Reich and Y. Shehu}, Relaxed inertial methods for solving the split monotone variational inclusion problem beyond co-coerciveness, Optimization, {\bf 72} (2023), 607-646.		

\bibitem {Kor}{\rm G. M. Korpelevich,} An extragradient method for finding sadlle points and for other problems, {\it Ekon. Mat. Metody,} {\bf 12} (1976), 747-756.
		
\bibitem{manb} {\rm  P. E. Maing\'e,} Projected subgradient techniques and viscosity methods for optimization with variational inequality constraints, {\it Eur. J. Oper. Res.,} {\bf 205} (2010), 501-506.

\bibitem{mana}{\rm P. E. Maing\'e,} Strong convergence of projected subgradient methods for nonsmooth and nonstrictly convex minimization, {\it Set-Valued Anal.,} {\bf 16} (2008), 899–912.
\bibitem{mou}{\rm A. Moudafi,}  Proximal point algorithm extended to equilibrum problem, {\it J. Nat. Geom.,}  {\bf 15} (1999), 91-100.
\bibitem{mouda}{\rm  A. Moudafi,} Second-order differential proximal methods for equilibrium problems, {\it J. Inequal. Pure Appl. Math.,} {\bf 4} (1) (2003), 1-7.

\bibitem{PA4} {\rm G.N. Ogwo, C. Izuchukwu, Y. Shehu and O.T. Mewomo,} Convergence of relaxed inertial subgradient extragradient methods for quasimonotone variational inequality problems, Journal of Scientific Computing, {\bf 90} (2022), 1-36. 

		
\bibitem{opial}{\rm Z. Opial,} Weak convergence of successive approximations for nonexpansive mappings, {\it Bull. Amer. Math. Soc.,} {\bf 73} (1967), 591-597.	

\bibitem{jwp}{\rm J. W. Peng, Y. C. Liou, J. C. Yao,} An iterative algorithm combining viscosity method with parallel method for a generalized equilibrium problem and strict pseudocontractions, {\it Fixed Point Theory Appl.,} (2009), Article ID 794178, 21 p., DOI: 10.1155/2009/794178.
		
\bibitem{ppo}{\rm  J. Peypouquet,} Convex Optimization in Normed Spaces, Theory, Methods and Examples, Springer, Berlin (2015).
		
		
\bibitem{pol}{\rm B. T. Polyak,} Some methods of speeding up the convergence of iteration methods, {\it U.S.S.R. Comput. Math. and Math. Phys.,} {\bf 4} (5) (1964), 1-17.
		
\bibitem{xq}{\rm X. Qin, Y. J. Cho, S. M. Kang,} Convergence theorems of common elements for equilibrium problems and fixed point problems in Banach spaces, {\it J. Comput. Appl. Math.,} {\bf 225} (2009), 20-30.
		
\bibitem{qu}{\rm  T. D. Quoc, L. D. Muu, V. H. Nguyen,} Extragradient algorithms extended to equilibrium problems, {\it Optimization} {\bf57} (2008), 749-776.
		
\bibitem{rp}{\rm H. Rehman, P. Kumam, Y. Je Cho, Y. I. Suleiman, W. Kumam, }  Modified Popov's explicit iterative algorithms for solving pseudomonotone equilibrium problems, {\it Optim. Methods and Softw.,} {\bf36}(1) (2021), 82-113.
\bibitem{hur}{\rm  H. U. Rehman, P. Kumam, M. Shutaywi, N. Pakkaranang, N. Wairojjana,} An inertial extragradient method for iteratively solving equilibrium problems in real Hilbert spaces, {\it  International J. Comp. Mathematics,} {\bf 99} (6) (2020), 1081-1104.


\bibitem{YCQY} {\rm Y. Shehu, C. Izuchukwu, J-C Yao and X. Qin,} Strongly convergent inertial extragradient type method for equilibrium problems, Applicable Analysis, (2021), http://dx.doi.org/10.1080/00036811.2021.2021187


\bibitem{Shehu}{\rm Y. Shehu, O.S. Iyiola, X-H. Li, Q-L. Dong,} Convergence analysis of projection method for variational inequalities, {\it Set-Valued Var. Anal.}, {\bf 29} (2021), 145-161.


\bibitem{thong}{\rm D. V. Thong, P. Cholamjiak, M. T.  Rassias, Y. J. Cho,} Strong convergence of inertial subgradient extragradient algorithm for solving pseudomonotone equilibrium problems, {\it  Optim. Letters,} {\bf 16} (2) (2022), 545-573.

\bibitem{vm}{\rm  N. T. Vinh, L. D. Muu,} Inertial extragradient algorithms for solving equilibrium problems, {\it Acta Mathematica Vietnamica,} {\bf  44} (3) (2019), 639-663.
\bibitem{tta}{\rm  J.Yang, H. Liu,} The subgradient extragradient method extended to pseudomonotone equilibrium problems and fixed point problems in Hilbert space, {\it Optim. Lett.,} {\bf 14,}  (2020), 1803–1816.

\end{thebibliography}
\end{document}